\documentclass[11pt,a4paper]{article}
\usepackage[top=2.6cm,bottom=2.6cm,left=2.6cm,right=2.6cm]{geometry}
\usepackage[english]{babel}
\usepackage[utf8]{inputenc}
\usepackage[T1]{fontenc}
\usepackage{amsmath,amssymb,amsthm,mathtools,enumerate,color}
\usepackage[hidelinks]{hyperref}
\numberwithin{equation}{section}

\newtheorem{theorem}{Theorem}[section]
\newtheorem*{theorem*}{Theorem}
\newtheorem{proposition}[theorem]{Proposition}
\newtheorem{lemma}[theorem]{Lemma}
\newtheorem{corollary}[theorem]{Corollary}
\theoremstyle{definition}

\newtheorem{remark}[theorem]{Remark}
\newtheorem*{remark*}{Remark}
\numberwithin{equation}{section}

\newcommand{\R}{\mathbb R}
\newcommand{\C}{\mathbb C}

\newcommand{\Ker}{\mathrm{Ker}}
\renewcommand{\ker}{\mathrm{Ker}}
\newcommand{\im}{\operatorname{Im}}

\date{October 30, 2026}

\begin{document}
\title{A Direct Polynomial Approach to Spectral Decomposition}
\author{Simon Bossoney and Marc Troyanov\\
\small École Polytechnique Fédérale de Lausanne\footnote{simon.bossoney@epfl.ch, marc.troyanov@epfl.ch}}

\maketitle

\begin{abstract}
We give a direct construction of the spectral projectors of a complex square matrix $A$ based on explicit Hermite interpolation polynomials introduced in
\cite{BT1}. Evaluating these polynomials at the matrix yields the spectral decomposition $A=\sum_{i=1}^r(\lambda_iP_i+N_i)$,  from which the Primary Decomposition Theorem, the Cayley--Hamilton theorem, criteria for diagonalizability, the spectral theorem for normal matrices, and an explicit functional calculus for matrices are obtained by short and direct arguments. 
The same construction, extended to perfect fields by an invariance
argument under the automorphisms of the splitting field, also yields the Jordan--Chevalley decomposition.

\medskip 

\noindent 


\medskip 

\noindent 
AMS Subject Class:  {Primary 15A21, 15A24; Secondary 41A05}
\end{abstract}


\section{Introduction}
Several fundamental results in linear algebra can be derived from the spectral decomposition of a matrix. The goal of this article is to show that the corresponding spectral projectors admit a particularly simple polynomial construction, from which a large part of the classical structure theory of matrices follows directly.

The construction is based on the \emph{Hermite interpolation polynomials}\footnote{Note that Verde-Star \cite{VerdeStar} uses the same terminology for a related, but different, family of interpolation polynomials.} introduced in \cite {BT1}:
\begin{equation}\label{defhi}
 h_i(t)=1-\left(1-\prod_{j\ne i}\left(\frac{t-\lambda_j}{\lambda_i- \lambda_j}\right)^{m_j} \right)^{m_i}
\end{equation}
for some distinct $\lambda_1, \dots \lambda_r \in \C$ and multiplicities $m_i\in \mathbb{N}$. These polynomials
are the natural analogs of the classical Lagrange polynomials for interpolation with higher multiplicities;  they satisfy the congruences
$$
 h_i(t)\equiv\delta_{ij}\ \bmod{(t-\lambda_j)^{m_j}}.
$$
Given a matrix $A\in M_n(\C)$ with minimal polynomial $\mu_A(t)=\prod_{i=1}^r(t-\lambda_i)^{m_i}$, evaluating these polynomials at $A$ gives the spectral decomposition
$$
 A=\sum_{i=1}^r(\lambda_iP_i+N_i),
$$
where $P_i=h_i(A)$ are the spectral projectors and $N_i=(A-\lambda_i I_n)P_i$ are the nilpotent components (see Theorem \ref{thm:spectral-decomp}).
From this decomposition, we obtain direct proofs of the Primary Decomposition Theorem, the Cayley--Hamilton theorem, criteria for diagonalizability, and the spectral theorem for normal matrices. The same polynomial construction also provides a natural approach to matrix functional calculus: combining the spectral projectors with the Taylor polynomials of a function at the eigenvalues yields an explicit formula for the corresponding matrix function, without recourse to the Jordan canonical form. In particular, we recover the classical Schwerdtfeger formula and, for holomorphic functions, the Cauchy integral representation.

In the final part of the paper, we extend our consideration to matrices over an arbitrary perfect field. Although the individual spectral projectors need no longer be defined over the ground field, the polynomial producing the semisimple part of the decomposition is shown to have coefficients in the ground field. This immediately gives the Jordan--Chevalley decomposition.

Polynomial spectral projectors have long been used in linear algebra, for instance in the proof of the Primary Decomposition Theorem given by Hoffman and Kunze \cite{HoffmanKunze}. The distinguishing feature of the present approach is the explicit construction of these projectors by means of the elementary polynomials \eqref{defhi}. 

\section{Spectral resolution of a complex square matrix}\label{sec:spectral}

\subsection{Setting and main Theorem}
For simplicity, we begin with the classical setting of a complex square matrix.  Given a nonzero matrix $A\in M_n(\C)$, we denote by $\C[A]$ the unital subalgebra of $M_n(\C)$ generated by $A$,
that is $\C[A]=\operatorname{span}\{A^k\mid k\ge0\}$.  This is a commutative algebra and the evaluation map
$$
\Phi_A:\C[t]\longrightarrow\C[A],\qquad p\longmapsto p(A),
$$
is a surjective algebra homomorphism. Because $\dim(\C[A]) <\infty$ and  $\dim\C[t]=\infty$, this map is not injective, and we have the isomorphism 
$$
 \mathbb{C}[t]/\ker(\Phi_A) \cong \mathbb{C}[A].
$$
The nonzero elements in $\Ker(\Phi_A)$ are the \emph{annihilating polynomials} of $A$. Among them, there exists a unique monic polynomial of minimal degree, called the \emph{minimal polynomial} of $A$ and denoted by $\mu_A$. Using Euclidean division, one sees  that every annihilating polynomial is a multiple of $\mu_A$. Equivalently,
$$
 \ker(\Phi_A) = \left( \mu_A\right)
$$
is the principal ideal generated by $\mu_A$. The central result of this paper is the following 
\begin{theorem}[Spectral decomposition of a complex matrix]\label{thm:spectral-decomp}
Let $\mu_A(t)=\prod_{i=1}^r(t-\lambda_i)^{m_i}$ be the complete factorization of the minimal polynomial of $A$. Then $A$  can be written as 
\begin{equation}\label{spec.dec}
    A=\sum_{i=1}^r(\lambda_iP_i+N_i),
\end{equation}
where $P_i,N_i\in\C[A]$ are defined by
\begin{equation}\label{defNP}
 P_i = h_i(A), \qquad  N_i = (A-\lambda_i) h_i(A). 
\end{equation}

Furthermore
\begin{equation}\label{propNP}
 P_iP_j=\delta_{ij}P_i,\qquad \sum_{i=1}^rP_i=I_n,\qquad  N_i^k=(A-\lambda_iI_n)^kP_i\quad(k\ge1),
\end{equation}
and each $N_i$ is nilpotent of order $m_i$.
\end{theorem}
The proof is given in the next section. The uniqueness of the decomposition \eqref{spec.dec} is established in Corollary \ref{cor:uniqueness} below.
 
 \smallskip 
 
The matrices $P_i$ are called the \emph{spectral projectors} of $A$,\footnote{The term \emph{Frobenius covariants} is also used in the literature for these projectors; see, for example, \cite{Rinehart}.} while the matrices $N_i$ are the corresponding \emph{nilpotent components}. Since they all belong to the commutative algebra $\C[A]$, they commute pairwise:
\begin{equation}\label{commutativity}
 [P_i,P_j]=[P_i,N_j]=[N_i,N_j]=0,
\end{equation}
where we use the notation $[S,T] = ST-TS$. The above theorem also implies that $N_iP_j = \delta_{ij}N_i$, and  $N_iN_j = 0$ if $i\neq j$ 

\smallskip 

For the proof, we will use the  following result on the polynomials 
defined in \eqref{defhi}:
\begin{proposition}\label{proprieteshi}
The polynomial $h_i$ satisfies the following congruence modulo $(t-\lambda_j)^{m_j}$\emph{:}
\begin{equation}\label{eq:duality}
 h_i\equiv \delta_{ij}\  \bmod{(t-\lambda_j)^{m_j}}.
\end{equation}
Moreover, the following congruences modulo $\mu_A$ hold: 
\begin{equation}\label{eq:idempotent-relations}
  h_i^2\equiv h_i \quad (i=1,\ldots,r),
  \qquad
  h_ih_j\equiv 0 \quad (i\ne j),
  \qquad
  \sum_{i=1}^r h_i\equiv 1.
\end{equation}
\end{proposition}
The proof is given in \cite[Proposition~2.1 and Corollary~2.3]{BT1}.  We will also need the following elementary lemma.
\begin{lemma}\label{lem.anh}
For every $i$ and every integer $k\ge0$, the polynomial $(t-\lambda_i)^k h_i(t)$  is divisible by $\mu_A$ if and only if $k\ge m_i$. In particular, for $k\ge m_i$ it is an annihilating polynomial of $A$.
\end{lemma}

\begin{proof}
Set $q_i(t)=\prod_{j\ne i}\left(\frac{t-\lambda_j}{\lambda_i-\lambda_j}\right)^{m_j}$ and  observe that 
$$
 h_i = 1-(1-q_i)^{m_i}=q_i\theta_i,
$$
for some polynomial $\theta_i$. It follows that for any $k\ge m_i$ the polynomial 
$(t-\lambda_i)^{k}h_i(t)$ is divisible by $(t-\lambda_i)^{m_i}q_i(t)$, which is a nonzero multiple of $\mu_A$.

To prove the converse direction, we observe that $h_i(\lambda_i)=1-(1-q_i(\lambda_i))^{m_i}=1$ since  $q_i(\lambda_i)=1$. It follows that  $(t-\lambda_i)$ does not divide $h_i$; therefore $(t-\lambda_i)^k h_i(t)$ vanishes  at $\lambda_i$  with multiplicity exactly $k$ and is thus not a multiple of $\mu_A$ if  $k<m_i$.
\\
\end{proof}

Note that each polynomial $h_i$ may be replaced by its remainder modulo
$\mu_A$ while preserving Proposition~\ref{proprieteshi} and Lemma~\ref{lem.anh}, which are the only properties of $h_i$ used in the sequel. This reduction is advantageous from
a computational viewpoint, since the degree of $h_i$ as defined by
\eqref{defhi} can be significantly larger than $\deg(\mu_A)$.

\subsection{Proof of Theorem \ref{thm:spectral-decomp}}
By definition \eqref{defNP},  we have $P_i=h_i(A)$  and $N_i=(A-\lambda_i)P_i$, so that Proposition~\ref{proprieteshi} implies 
$$
 P_iP_j=(h_ih_j)(A)=\delta_{ij}P_i \quad \text{and} \quad 
 \sum_{i=1}^rP_i = \sum_{i=1}^rh_i(A) =\left(\sum_{i=1}^rh_i\right)(A)=I_n.
$$
Moreover $AP_i =\lambda_iP_i + N_i$; we thus obtain the following decomposition for $A$:
$$
 A=A\sum_{i=1}^rP_i=\sum_{i=1}^rAP_i =\sum_{i=1}^r(\lambda_iP_i+N_i).
$$
Since $P_i^k = P_i$ for every $k\ge 1$, we also have
$$
 N_i^k=\bigl((A-\lambda_iI_n)h_i(A)\bigr)^k = (A-\lambda_iI_n)^kP_i^k = (A-\lambda_iI_n)^kP_i.
$$
Finally, Lemma~\ref{lem.anh} shows that $(t-\lambda_i)^kh_i(t)$ is an annihilating polynomial of $A$ if and only if $k\ge m_i$. Hence $N_i^k=\bigl((t-\lambda_i)^kh_i\bigr)(A)$ vanishes if and only if $k\ge m_i$, i.e. $N_i$ is nilpotent of order exactly $m_i$. This completes the proof of  Theorem~\ref{thm:spectral-decomp}. 

\hfill  \qedsymbol
 
\medskip

\begin{remark}
The spectral decomposition \eqref{spec.dec} is a classical result; see, for example, \cite[Chapter~VII, \S1]{DunfordSchwartz}. Verde-Star \cite{VerdeStar} also derives it using Hermite interpolation, but with interpolation polynomials constructed recursively via divided differences. Our approach yields the explicit closed-form formula \eqref{defNP} for the spectral projectors and nilpotent components, based directly on the polynomials $h_i$.
\end{remark}


\section{Applications of the spectral decomposition}\label{sec:consequences}

In this section, we show how the spectral resolution of
Theorem~\ref{thm:spectral-decomp} provides a unified and direct approach to
several classical results of linear algebra, including the Primary
Decomposition Theorem, the Cayley--Hamilton Theorem, criteria for
diagonalizability and the spectral theorem for normal matrices.
This reverses the usual order of presentation found in
many textbooks, where the Primary Decomposition Theorem is established first
and the spectral projectors are introduced later (or never).

The only part of the classical theory not recovered by this approach
is the structure of the nilpotent components. Its description
requires the construction of Jordan chains and leads to the Jordan canonical
form.

\subsection{Primary decomposition}

For each root $\lambda_i$ of $\mu_A$, we define the associated  \emph{generalized eigenspace} (also called the \emph{characteristic subspace}) as
$$
 U_i=\ker(A-\lambda_iI)^{m_i} \subset \C^n.
$$

\begin{corollary}[Primary Decomposition Theorem]\label{cor:primary}
For every $i$ we have $\im(P_i)=U_i\neq\{0\}$, and
\begin{equation}
 \mathbb C^n=U_1\oplus\cdots\oplus U_r.
\end{equation}
In particular, a basis of $U_i$ is obtained by choosing $d_i=\dim U_i$ linearly independent columns of $P_i$.
\end{corollary}

\begin{proof}
We first show that $P_i\neq0$. Indeed, by Lemma~\ref{lem.anh} with $k=0$, $h_i$ is not a multiple of $\mu_A$, hence not an annihilating polynomial of $A$, therefore $P_i=h_i(A)\neq0$.

\smallskip

We now show that $\im(P_j) = U_j$. The inclusion 
$\im(P_j)\subset U_j$ follows from  $(A-\lambda_jI)^{m_j}P_j=N_j^{m_j}=0$. 
To prove the converse inclusion $U_j \subset \im(P_j)$, pick $x\in U_j=\operatorname{Ker}(A-\lambda_jI)^{m_j}$. Since
$h_i\equiv\delta_{ij}\bmod(t-\lambda_j)^{m_j}$, we can write
$h_i(t)-\delta_{ij}=(t-\lambda_j)^{m_j}\theta(t)$ for some polynomial
$\theta$, so that
$$
 P_ix-\delta_{ij}x = (h_i(A) - \delta_{ij})(x) =\theta(A)(A-\lambda_jI)^{m_j}x=0 .
$$
Hence $P_ix=\delta_{ij}x$ for every $i$, and $x=\sum_iP_ix=P_jx\in\im(P_j)$.

\smallskip

Using now $\sum_jP_j=I$ and $P_iP_j=0$ for $i\neq j$, we have 
$$
 U_1 \oplus\cdots\oplus U_r = \im(P_1)\oplus\cdots\oplus\im(P_r) = \C^n.
$$
Finally $U_j=\im(P_j)\neq\{0\}$ because $P_j\neq0$.
\end{proof}

\bigskip

\begin{corollary}
The eigenvalues of $A$ are precisely the roots of $\mu_A$ and the corresponding eigenspaces are 
$$
  E_{\lambda_i} =\ker(A-\lambda_iI) =\ker( N_i)  \cap U_i.
$$
\end{corollary}

\begin{proof}
Since $N_i=(A-\lambda_iI)P_i$ and $P_i$ acts as the identity on $U_i$, we have
$$
N_i|_{U_i}=(A-\lambda_iI)|_{U_i}.
$$
Hence
$$
\ker(N_i)\cap U_i =\ker(A-\lambda_iI)\cap U_i =\ker(A-\lambda_iI),
$$
because $\ker(A-\lambda_iI)\subseteq U_i$.

It remains to show that the eigenvalues of $A$ are exactly the roots of $\mu_A$. We first show that each $\lambda_i$ is an eigenvalue: indeed, $U_i\neq\{0\}$ and $(A-\lambda_iI)^{m_i}=0$ on $U_i$, so the restriction of $A-\lambda_iI$ to $U_i$ is nilpotent, hence has a nonzero kernel.

Conversely, if $Ax=\lambda x$ with $x\neq 0$, then
$0=\mu_A(A)x=\mu_A(\lambda)x$, so $\mu_A(\lambda)=0$.
\end{proof}

We next show that the decomposition \eqref{spec.dec} in Theorem~\ref{thm:spectral-decomp} is unique in the following sense:
\begin{corollary}\label{cor:uniqueness}  
Let $\lambda_1,\dots,\lambda_r$ be the distinct roots of $\mu_A$. Suppose
that
$$
 A=\sum_{i=1}^r\bigl(\lambda_iP_i'+N_i'\bigr),
$$
where the $P_i'$ are nonzero projectors, and the $N_i'$ are nilpotent with 
$$
 P_i'P_j'=\delta_{ij}P_i',  \quad  \sum_iP_i'=I, \quad  
 P_i'N_i'=N_i'P_i'=N_i'.
$$
Then $P_i'=P_i$ and $N_i'=N_i$ for all $i$.
\end{corollary}

\begin{proof}
Let $U_i'=\im(P_i')$, so that $\C^n=U_1' \oplus \cdots \oplus U_r'$. We divide the proof in three steps: 

\smallskip
\emph{Step 1: $A$ acts on $U_i'$ as $(\lambda_i+N_i')$.}
For $j\neq i$ we have $N_i'P_j'=N_i'P_i'P_j'=0$, hence
$AP_i'=\lambda_iP_i'+N_i'$. Thus $U_i'$ is $A$-invariant, and
$A-\lambda_i$ is nilpotent on $U_i'$, while for $j\neq i$ the map
$A-\lambda_j=(\lambda_i-\lambda_j)+N_i'$ is invertible on $U_i'$.

\smallskip
Step 2: $U_i\subseteq U_i'$. Let $x\in U_i$, so that
$(A-\lambda_i I)^{m_i}x=0$. Write $x=\sum_j P_j'x$.
Since each $U_j'$ is $A$-invariant, we obtain
$(A-\lambda_i I)^{m_i}P_j'x=0$ for every $j$.
For $j\ne i$, the restriction of $A-\lambda_i I$
to $U_j'$ is invertible by Step 1, hence $P_j'x=0$.
Therefore $x=P_i'x\in U_i'$.

\smallskip

Step 3: Conclusion. Since both $\bigoplus_i U_i$ and
$\bigoplus_i U_i'$ equal $\mathbb C^n$, the inclusions
$U_i\subseteq U_i'$ imply $U_i=U_i'$ for every $i$.
Thus $P_i'=P_i$ by Corollary~\ref{cor:primary},
and $N_i'=(A-\lambda_i I)P_i'=N_i$.
\smallskip
\end{proof}

\smallskip

\begin{remark}
Using spectral projectors in the proof of the Primary Decomposition Theorem is not a new idea; see for example \cite[Chapter~6, Theorem~12]{HoffmanKunze}. In this reference, the projectors are constructed by means of Bézout identities, and the Primary Decomposition Theorem is seen as the main structural result, rather than a consequence of the spectral resolution.
\end{remark}

\subsection{The Cayley--Hamilton theorem}

\begin{corollary}\label{cor:eigenvalues-CH}
The minimal polynomial $\mu_A$ divides the characteristic polynomial
$\chi_A(t)=\det(tI-A)$. In particular, $\chi_A(A)=0$.
\end{corollary}

\begin{proof}
By Theorem \ref{thm:spectral-decomp} and Corollary~\ref{cor:primary}, in a basis adapted to the decomposition
$\mathbb C^n=U_1\oplus\cdots\oplus U_r$, the matrix $A$ is similar to
$A_1\oplus\cdots\oplus A_r$, where $A_i\in M_{d_i}(\mathbb C)$
is the submatrix corresponding to $U_i$ and $d_i=\dim U_i$.
Since $U_i=\ker(A-\lambda_iI)^{m_i}$, the matrix
$B_i=A_i-\lambda_iI$ is nilpotent, so $A_i$ has $\lambda_i$ as its only
eigenvalue and $\chi_{A_i}(t)=(t-\lambda_i)^{d_i}$.

We claim that $m_i\le d_i$. If $m_i=1$ this follows from $U_i\neq\{0\}$.
If $m_i\ge2$, then $N_i^{m_i-1}=(A-\lambda_iI)^{m_i-1}P_i\neq0$ by
Theorem \ref{thm:spectral-decomp}; as $P_i$ is the identity on $U_i$ and vanishes on the other $U_j$, this means $B_i^{m_i-1}\neq0$. Since a nilpotent matrix
of size $d_i$ satisfies $B_i^{d_i}=0$, we get $m_i-1<d_i$.

Therefore
$$
 \chi_A(t)=\prod_i\chi_{A_i}(t)=\prod_i(t-\lambda_i)^{d_i}
$$
is divisible by $\mu_A(t)=\prod_i(t-\lambda_i)^{m_i}$.
It follows that $\chi_A(A)=0$.
\end{proof}

\smallskip

Note that the above proof shows that the algebraic multiplicity $d_i$ of the eigenvalue $\lambda_i$ in the characteristic polynomial $\chi_A(t)$ is equal to $\mathrm{rank}(P_i) = \dim(U_i)$. 

\subsection{Diagonalizable matrices} \label{sec.diagonalizable}

Recall that a matrix $A\in M_n(\C)$ is \textit{diagonalizable} if there exists a basis of $\C^n$ consisting of eigenvectors. Equivalently, $\C^n$ is the direct sum of its eigenspaces.

\begin{corollary}\label{cor:diagonalizable}
Let $N_i$ and $N=\sum_{i=1}^r N_i$ be the nilpotent components of Theorem~\ref{thm:spectral-decomp}.
The following conditions are equivalent:
\begin{enumerate}[(i)]
\item $A$ is diagonalizable;
\item $m_i=1$ for every $i$ (i.e.\ $\mu_A$ has only simple roots);
\item $N_i=0$ for every $i$;
\item $N=\sum_i N_i =0$.
\end{enumerate}
\end{corollary}
Observe that condition (iii), together with Theorem \ref{thm:spectral-decomp}, tells us that $A$ is diagonalizable if and only if the spectral decomposition reduces to $A=\sum_{i=1}^r\lambda_iP_i$.

\begin{proof}
By Corollary \ref{cor:primary}, $A$ is diagonalizable if and only if
$U_i=\ker(A-\lambda_iI)^{m_i}$ coincides with the eigenspace $\ker(A-\lambda_iI)$
for every $i$, which holds if and only if $m_i=1$ for all $i$. This gives (i) $\Leftrightarrow$ (ii).

The equivalence (ii) $\Leftrightarrow$ (iii) follows from the fact that $N_i$ is nilpotent of order $m_i$.

If $N=0$, then $N_i=NP_i=0$ for every $i$; this shows that (iv) 
$\Rightarrow$ (iii). Finally (iii) $\Rightarrow$ (iv) is immediate.
\end{proof}

 \medskip 

\begin{remark} 
When $A$ is diagonalizable, we have $m_i = 1$ for each $i$ and the polynomials $h_i$ reduce to the \textit{Lagrange polynomials}
\begin{equation}\label{eq:Lagrange}
  \ell_i(t)=\prod_{j\ne i}\frac{t-\lambda_j}{\lambda_i-\lambda_j},
\end{equation}
In that case, the spectral decomposition of $A$ takes the classical \textit{Sylvester formula} $A=\sum_{i=1}^r\lambda_i\ell_i(A)$.
\end{remark}

\subsection{Simultaneous diagonalization} \label{sec.simultane}

If $A', A''\in M_n(\C)$ are commuting diagonalizable matrices with spectral decompositions
$$A'=\sum_{i=1}^r \lambda'_i P_i, \qquad A''=\sum_{j=1}^s \lambda''_j Q_j,$$
then the projectors $P_i$ and $Q_j$ commute, since they are polynomials in $A'$ and $A''$
respectively, and $[A',A''] = 0$. The matrices $R_{ij}=P_iQ_j = Q_jP_i$  are then 
projectors satisfying\footnote{Note that we do not claim here that the $R_{ij}$ are all non zero.}
$$R_{ij}^2=R_{ij}, \qquad R_{ij}R_{i'j'}=0 \ \text{ if }\ (i',j')\ne(i,j), \qquad
\sum_{i,j}R_{ij}=I,$$
and we have
$$
 A'=\sum_{i,j}\lambda'_i R_{ij}, \qquad A''=\sum_{i,j}\lambda''_j R_{ij}.
$$
 This shows that $A'$ and $A''$ are simultaneously diagonalizable: any basis
obtained by choosing an arbitrary basis in each nonzero subspace $\im(R_{ij})$ is a common eigenbasis for both matrices.

\subsection{The case of normal matrices}
We now recover the spectral theorem for normal complex matrices from Theorem \ref{thm:spectral-decomp}. Recall that the \textit{adjoint} of a matrix $A\in M_n(\C)$ is the matrix $A^*=\bar{A}^\top$. It is the unique matrix such that 
$$
 \langle A x,y\rangle =  \langle x,A^*y\rangle 
$$
for all $x,y\in \C^n$, where $ \langle x,y\rangle  = \sum_i x_i \overline{y}_i$ is the standard Hermitian scalar product in $\C^n$. The matrix $A$ is \emph{normal} if $[A,A^*]=0$; it is selfadjoint if $A^* = A$.  We first recall some elementary properties of normal matrices.

\begin{lemma}\label{lem:normal-basic-v2}
\begin{enumerate}[(i)]
\item $A\in M_n(\C)$ is normal if and only if $\|Ax\|=\|A^*x\|$ for all $x\in\C^n$.
\item Every normal nilpotent matrix is zero.
\item Every normal projector is selfadjoint.
\end{enumerate}
\end{lemma}

\begin{proof}
Statement (i) follows from the identity $\|A^*x\|^2-\|Ax\|^2 =\langle[A,A^*]x,x\rangle$. 

For (ii), suppose $N\ne0$ is both normal and nilpotent, and let $m\ge2$ be minimal with $N^m=0$. By (i), we have $\ker (N) =\ker (N^*)$, so $\im (N^{m-1})\subseteq\ker (N^*)$ and $N^*N^{m-1}=0$. Then 
$$
  \|N^{m-1}x\|^2=\langle N^*N^{m-1}x,N^{m-2}x\rangle=0
$$
for every $x$, contradicting the minimality of $m$.

For (iii), let $P^2=P$ be normal. By (i), we have  $\ker(P)=\ker (P^*)$, thus $\ker\left(  P^*(I-P)\right)  = \ker\left(  P(I-P) \right) = \C^n$, that is $P^*(I-P) =0$. Similarly  $P(I-P^*)=0$, therefore 
$$
 P^* - P = P^*P - PP^* = 0.
$$
\end{proof}

\begin{theorem}[Spectral theorem for normal matrices]\label{thm:normal-spectral}
Let $A\in M_n(\C)$ be normal. Then $A=\sum_{i=1}^r\lambda_iP_i$,
where $\lambda_1,\ldots,\lambda_r$ are the distinct eigenvalues of $A$ and the projectors satisfy
$$P_i=P_i^*, \qquad P_iP_j=\delta_{ij}P_i, \qquad \sum_{i=1}^rP_i=I.$$
Moreover $P_i=\ell_i(A)$, where $\ell_i$ is the Lagrange polynomial \eqref{eq:Lagrange},
and $\C^n$ admits an orthonormal basis of eigenvectors of $A$.
\end{theorem}

\begin{proof}
By Theorem~\ref{thm:spectral-decomp}, $A=\sum_i(\lambda_iP_i+N_i)$.
Since $A$ is normal, every polynomial in $A$ is normal; in particular each $N_i$ is normal and nilpotent, so $N_i=0$ by Lemma~\ref{lem:normal-basic-v2}. As explained in \S \ref{sec.diagonalizable}, this implies that $A$ is diagonalizable and we have  $A=\sum_i\lambda_iP_i$. Moreover, the projectors are given by the Lagrange interpolation polynomials:
$$
 P_i = h_i(A) = \ell_i(A),
$$
since each $m_i = 1$. Also, each spectral projector $P_i\in\C[A]$ is normal, hence $P_i=P_i^*$ by Lemma~\ref{lem:normal-basic-v2}. 
The subspaces $\im(P_i)$ are therefore pairwise orthogonal, and choosing an orthonormal basis
in each of them yields an orthonormal basis for the eigenvectors of $A$.
\end{proof}

\medskip 
 {\small 
\begin{remark}
Theorem \ref{thm:normal-spectral} extends to the case of matrices with coefficients in a   field $(K,*)$ equipped with a \emph{positive involution}, that is, an automorphism $*$ of
order at most $2$, and such that $\sum_i x_i^* x_i = 0 \Leftrightarrow x_i = 0$ for each $i$. Such fields are classified by the Artin-Schreier Theory; we refer to \cite{KMRT,Scharlau} for a systematic treatment. The main examples are $\R$ where $*$ is the identity,  and $\C$ where $*$ is the complex
conjugation, together with all their subfields stable under complex conjugation. In this general setting, the adjoint of a matrix $A\in M_n(K)$ is defined by $(A^*)_{ij} = (A_{ji})^*$,  and the statement and proof of  Theorem~\ref{thm:normal-spectral}  carry through, assuming the characteristic polynomial of $A$ splits over $K$, except
that one obtains an orthogonal (not necessarily orthonormal) basis of
eigenvectors, since normalizing requires square roots of elements.
This covers in particular the case of real symmetric matrices, since in that case the 
characteristic polynomial always splits over $\R$.
\end{remark}
}

\subsection{Application to matrix functional calculus}

Matrix functions are a classical topic in matrix analysis; standard
references are \cite[Chap.~6]{HornJohnson}, \cite[Chap.~5]{Gantmacher} and
\cite{Higham}. See \cite{Rinehart} for a historical account up to the
1950s, and \cite[Chap.~VII]{DunfordSchwartz} for the functional analytic
viewpoint.
In this section, we show how the spectral decomposition gives a direct
formula for functions of a matrix, thereby avoiding the reduction to the
Jordan canonical form often used in computing functions of matrices. 

\smallskip

Let $g$ be a function defined on an open neighborhood of the spectrum
$\{\lambda_1,\dots,\lambda_r\}$ of $A\in M_n(\C)$, with derivatives up to
order $m_i-1$ at each $\lambda_i$.\footnote{In particular, $g$ may be taken
analytic, but analyticity is not required.} For each $i$, let
\begin{equation}\label{defgi}
  g_i(t)=\sum_{k=0}^{m_i-1}\frac{g^{(k)}(\lambda_i)}{k!}(t-\lambda_i)^k
\end{equation}
be the Taylor polynomial of order $m_i-1$ of $g$ at $\lambda_i$.

By the Hermite interpolation formula of Theorem~3.1 in \cite{BT1}, the unique
polynomial $f$ of degree less than $\deg(\mu_A)$ satisfying
$$f(t)\equiv g_i(t)\ \bmod{(t-\lambda_i)^{m_i}},\qquad i=1,\dots,r,$$
is the remainder modulo $\mu_A$ of $F(t)=\sum_{i=1}^r g_i(t)h_i(t)$.
Since $\mu_A(A)=0$, this leads to the following definition:
\begin{equation}\label{def.gA}
  g(A):=F(A)=\sum_{i=1}^r g_i(A)h_i(A).
\end{equation}
In particular, $g(A)$ depends only on the derivatives $g^{(k)}(\lambda_i)$, for $0\leq k<m_i$, and not on the values of $g$ away from the spectrum.

\smallskip 

The following formula, which is classically known as the \textit{Schwerdtfeger formula},  expresses $g(A)$ in terms of the spectral decomposition.
\begin{proposition}\label{prop.calcfonc}
With the above notation, we have 
\begin{equation}\label{{eq:Schwerdtfeger}}
  g(A)=\sum_{i=1}^r\sum_{k=0}^{m_i-1}\frac{g^{(k)}(\lambda_i)}{k!}N_i^kP_i,
\end{equation}
where $A=\sum_{i=1}^r(\lambda_iP_i+N_i)$ is the spectral decomposition of
Theorem~\ref{thm:spectral-decomp}, with the convention $N_i^0=I$.
\end{proposition}

\begin{proof}
From \eqref{defgi} we have
$$
g_i(A)=\sum_{k=0}^{m_i-1}\frac{g^{(k)}(\lambda_i)}{k!}(A-\lambda_iI)^k.
$$
Moreover, $h_i(A)=P_i$ and
$$
(A-\lambda_iI)^kP_i=N_i^kP_i
$$
for every $k\ge0$. Indeed, for $k=0$, both sides are equal to $P_i$, while for $k\ge1$, Theorem~\ref{thm:spectral-decomp} gives
$$
(A-\lambda_iI)^kP_i=N_i^k,
$$
and  $N_i^k=N_i^kP_i$. Substituting these identities into \eqref{def.gA} gives \eqref{{eq:Schwerdtfeger}}.
\end{proof}

\medskip

The following examples illustrate Proposition~\ref{prop.calcfonc}.
\begin{enumerate}[(1)]
\item The matrix exponential satisfies
$$\mathrm{e}^{sA}=\sum_{i=1}^r\sum_{k=0}^{m_i-1}
  \frac{s^k}{k!}\mathrm{e}^{s\lambda_i}N_i^kP_i.$$

\item For any $\alpha\notin\{\lambda_1,\dots,\lambda_r\}$, the resolvent is
$$(\alpha I-A)^{-1}=\sum_{i=1}^r\sum_{k=0}^{m_i-1}
  \frac{1}{(\alpha-\lambda_i)^{k+1}}N_i^kP_i.$$

\item If $A$ has no eigenvalue in $(-\infty,0]$, then using the principal
branch of the logarithm, we have
$$
 \log(A)=\sum_{i=1}^r\left(\log(\lambda_i)
 +\sum_{k=1}^{m_i-1}\frac{(-1)^{k-1}}{k\lambda_i^k}N_i^k\right)P_i.
$$

\item If $\alpha>0$ and $A$ has no eigenvalue in $(-\infty,0]$, then, using
the principal branch of $z^\alpha$,
$$
A^\alpha=\sum_{i=1}^r\sum_{k=0}^{m_i-1}
\binom{\alpha}{k}\lambda_i^{\alpha-k}N_i^kP_i,
$$
where $\displaystyle\binom{\alpha}{k} =\frac{1}{k!}\alpha(\alpha-1)\cdots(\alpha-k+1)$.
If $\alpha\in\mathbb N$, this formula holds without restriction on the eigenvalues. 
If $\alpha=\frac12$, we obtain the following formula for the principal square root of $A$:
$$
\sqrt{A}=\sum_{i=1}^r\sum_{k=0}^{m_i-1}\binom{1/2}{k}\lambda_i^{\frac12-k}N_i^kP_i
=\sum_{i=1}^r\sqrt{\lambda_i}\left(I+\frac{N_i}{2\lambda_i}-\frac{N_i^2}{8\lambda_i^2}
+\frac{N_i^3}{16\lambda_i^3}-\cdots\right)P_i.
$$
\end{enumerate}

\medskip 

Using Example~(2), one recovers the Cauchy integral formula for the holomorphic functional calculus. 
Namely, if $g$ is holomorphic in some neighborhood of the closure of a Jordan domain $D\subset\C$  with rectifiable boundary containing $\{\lambda_1,\ldots,\lambda_r\}$, then
\begin{equation}\label{holgA}
  g(A)=\frac{1}{2\pi i}\int_{\partial D}g(z)(zI-A)^{-1}\,dz.
\end{equation}
Indeed, substituting the expression for the resolvent from Example~(2) and applying Cauchy's integral formula for holomorphic functions, we have
\begin{align*}
\frac{1}{2\pi i}\int_{\partial D}g(z)(zI-A)^{-1}\,dz
&=\sum_{i=1}^r\sum_{k=0}^{m_i-1}
  \frac{1}{2\pi i}\int_{\partial D}
  \frac{g(z)}{(z-\lambda_i)^{k+1}}\,dz\;N_i^kP_i\\
&=\sum_{i=1}^r\sum_{k=0}^{m_i-1}
  \frac{g^{(k)}(\lambda_i)}{k!}\,N_i^kP_i,
\end{align*}
which coincides with \eqref{{eq:Schwerdtfeger}}.

 \medskip

We conclude this section with a few historical remarks; see \cite{Rinehart} and \cite[\S 1.10]{Higham} for further information. 

(i) \
Matrix functions are often introduced via the Jordan canonical form. A less common approach is based on Hermite interpolation; see, for example, \cite[Definition~1.4]{Higham}. L.Verde-Star also defines $g(A)$ by Hermite interpolation, using divided differences; see \cite[formula~(4.7)]{VerdeStar}.
The present approach is in the same spirit, but is based on the  polynomials \eqref{defhi} and their fundamental properties. 

(ii) \ 
Formula \eqref{{eq:Schwerdtfeger}} is classical and appears in various forms in the literature. 
According to Rinehart, versions of this formula already appear in the work of H. Schwerdtfeger (1935) and H.  Richter (1951);  see formulas (2.12) and (2.16) in \cite{Rinehart}. It also appears in \cite[Chapter~VII, \S1]{DunfordSchwartz}. 
In the diagonalizable case, it goes back to Sylvester (1883).
The distinctive feature of the present approach is the explicit construction of Frobenius covariants, that is the spectral projectors $P_i=h_i(A)$, 
and the nilpotent components $N_i=(A-\lambda_iI)P_i$ based on the polynomials \eqref{defhi}, which makes the whole construction 
direct and well suited to explicit computations, as shown in the examples above.

(iii) \
The Cauchy integral formula \eqref{holgA} is attributed to E.~Cartan
(1928) by Rinehart \cite{Rinehart}. It became the starting point of the
holomorphic functional calculus developed by Dunford and Schwartz in
\cite[Chapter~VII]{DunfordSchwartz}.


\section{The Jordan--Chevalley decomposition}\label{JordanChevalley}
The preceding sections are written over $\C$. However, the same construction works over an arbitrary field $K$ whenever the minimal polynomial splits over $K$, that is, if $\mu_A(t)=\prod_i(t-\lambda_i)^{m_i}$ with each $\lambda_i\in K$. Then the same polynomials $h_i\in K[t]$ give the same spectral resolution
$A=\sum_i(\lambda_iP_i+N_i)$.

When the minimal polynomial \emph{does not} split over $K$, the individual
spectral projectors are  still defined over the splitting field of $\mu_A$. Nevertheless, if $K$ is perfect, the polynomial producing their weighted sum
$\sum_i\lambda_iP_i$ has all its coefficients in $K$. This yields the semisimple part in the Jordan--Chevalley decomposition over the ground field.
\begin{theorem}[Jordan--Chevalley decomposition]\label{thm:jordan-chevalley}
Let $K$ be a perfect field and $A\in M_n(K)$. Then there exists a unique decomposition
$$A=S+N,$$
where $S,N\in K[A]$, $S$ is semisimple, and $N$ is nilpotent.
\end{theorem}
Note that $S$ and $N$ commute, since they are polynomials in $A$.
Before proving this theorem, we shall review the definitions and  basic properties of perfect fields and semisimple matrices. 

\subsection{On perfect fields}\label{subsec:perfect}

We recall a few standard notions from field theory; see, for example, \cite{Lang}.

A \emph{splitting field} of a polynomial $p\in K[t]$ is a field extension $L$ of $K$ in
which $p$ splits into linear factors and which is generated over $K$ by the roots of $p$. It always exists and is unique up to isomorphism.

A polynomial $p\in K[t]$ is \emph{separable} if it has only simple roots in its splitting
field; equivalently, $p$ and its formal derivative $p'$ are coprime.

A field $K$ is \emph{perfect} if every irreducible polynomial in $K[t]$ is separable.
Every field of characteristic zero is perfect; a field of characteristic $p>0$ is perfect
if and only if every element admits a $p$-th root in $K$. In particular every finite field is perfect.

\subsection{A key interpolation polynomial}
Let $A\in M_n(K)$ and let $L$ be the splitting field of the minimal polynomial $\mu_A$. Over $L$, write $\mu_A(t)=\prod_{i=1}^r(t-\lambda_i)^{m_i}$ with $\lambda_i\in L$ pairwise distinct,
and let $h_i\in L[t]$ be the polynomials defined in \eqref{defhi}.
We define
\begin{equation}\label{defrho}
\rho(t)=\sum_{i=1}^r\lambda_i h_i(t). 
\end{equation}
By Proposition~\ref{proprieteshi}, we have
\begin{equation}\label{eq:rho-interpolation}
  \rho(t)\equiv\lambda_i\ \bmod{(t-\lambda_i)^{m_i}},\quad i=1,\ldots,r.
\end{equation}
The polynomial $\rho$ also has the following important property:
\begin{proposition}\label{prop.rhoinK} 
If $K$ is perfect, then $\rho\in K[t]$.
\end{proposition}

\begin{proof}
Let $L$ be the splitting field of $\mu_A$ over $K$, and let
$\sigma$ be a $K$-automorphism of $L$. Since $\sigma$ permutes
the roots of $\mu_A$ preserving their multiplicities, we may
write $\sigma(\lambda_i)=\lambda_{\sigma(i)}$. Definition \eqref{defhi} of the polynomials $h_i$'s then gives $\sigma(h_i)=h_{\sigma(i)}$. Hence
$$
\sigma(\rho)=\sum_i\sigma(\lambda_i)\sigma(h_i)
=\sum_i\lambda_{\sigma(i)}h_{\sigma(i)}=\rho.
$$
Every coefficient of $\rho$ is therefore fixed by all
$K$-automorphisms of $L$, and belongs to $K$ by
Proposition~\ref{prop:fixed} in the appendix. Thus $\rho\in K[t]$.
\end{proof}

\begin{remark}\label{rem:nonperfect-v2}  
The proposition fails for nonperfect fields. Let $K=\mathbb{F}_p(\alpha)$ with $\alpha$
transcendental and $\mu(t)=t^p-\alpha$. If $\beta^p=\alpha$, then over $L=K(\beta)$
we have $\mu(t)=(t-\beta)^p$, in particular $\beta$ is the unique root in the splitting field of $\mu$. In this example, we thus have $h_1=1$ and $\rho=\beta\notin K[t]$.
\end{remark}

\subsection{Semisimple matrices}\label{subsec:semisimple}

A matrix $A\in M_n(K)$ is \emph{simple} if the only $A$-invariant subspaces
of $K^n$ are $\{0\}$ and $K^n$ itself. It is \emph{semisimple}\footnote{Equivalently,
$A$ is semisimple if and only if the quotient algebra $K[A]\cong K[t]/(\mu_A)$
is semisimple, that is a direct product of fields. This  holds if and only if $\mu_A$ is squarefree.} if every $A$-invariant subspace of $K^n$ has an $A$-invariant complement.

If $K$ is perfect, $A$ is semisimple if and only if its minimal polynomial is separable, equivalently, $A$ is diagonalizable over the
splitting field of $\mu_A$.

\begin{lemma}\label{SSnilpotentisnul}
If $A\in M_n(K)$ is both nilpotent and semisimple, then $A=0$.
\end{lemma}

\begin{proof}
If $A\ne0$, then $\ker A\subset K^n$ is a proper invariant subspace with no invariant
complement.
\end{proof}

\begin{lemma}\label{lem:simultaneous-ss}
Let $K$ be a perfect field. If $A,B\in M_n(K)$ are commuting semisimple matrices, then $A\pm B$ are semisimple.
\end{lemma}

\begin{proof}
Let $L$ be a common splitting field of the minimal polynomials
of $A$ and $B$. Since $A$ and $B$ are commuting semisimple
matrices, they are simultaneously diagonalizable over $L$
(the argument of Section~\ref{sec.simultane} applies verbatim over the field $L$). 
Hence $A\pm B$ is diagonalizable over $L$ and therefore
semisimple over $K$.
\end{proof}

\subsection{Proof of Theorem~\ref{thm:jordan-chevalley}}
Let us define $S,N \in K[A]$ by 
$$
 S=\rho(A),\qquad N=A-\rho(A),
 $$
where $\rho \in K[t]$ is defined in \eqref{defrho}.  Theorem~\ref{thm:spectral-decomp} gives us the spectral decomposition $A=\sum_i(\lambda_iP_i+N_i)$ over the splitting field $L$ of $\mu_A$, and by the definition of $S$ and $N$, we have 
$$
  S=\sum_i\lambda_iP_i, \qquad N=\sum_iN_i.
 $$
Thus $S$ is diagonalizable over $L$, hence semisimple over $K$, and $N$ is nilpotent since each $N_i$ is nilpotent and $N_iN_j=0$ for $i\ne j$.

The argument proving uniqueness is classical; see \cite{Chevalley1951}. Let
$A=S'+N'$ be another such decomposition. Since all four matrices belong to
$K[A]$, they commute. The difference $S-S'$ is semisimple by
Lemma~\ref{lem:simultaneous-ss}, while $N'-N$ is nilpotent
because a sum of commuting nilpotent matrices is nilpotent. 
Since $S-S'=N'-N$, Lemma~\ref{SSnilpotentisnul} implies that both differences vanish.
This proves Theorem \ref{thm:jordan-chevalley}. 
\hfill\qedsymbol

\medskip

\begin{remark}
The essential step in the proof of Theorem~\ref{thm:jordan-chevalley} is
the construction of the interpolation polynomial $\rho$ satisfying
\eqref{eq:rho-interpolation}. It is first constructed over the splitting
field of the minimal polynomial and then shown, by an invariance argument
under the automorphisms of the splitting field, to have coefficients
in $K$. This is where the assumption that $K$ is perfect is used in the
construction of $S$ (Proposition~\ref{prop.rhoinK}).

Chevalley's original proof \cite{Chevalley1951} follows a different
argument. He constructs the semisimple part directly over $K$, without
factoring the minimal polynomial or computing its roots. A detailed
account of Chevalley's algorithm, together with an explicitly computed
example, is given in \cite{CoutyEsterleZarouf2011}. This paper also
contains an interesting historical account of the subject.
\end{remark}

\begin{remark}
In his classical paper \cite{Dunford1954}, Dunford introduced the class of
\emph{spectral operators} on complex Banach spaces by means of a resolution of
the identity satisfying suitable compatibility conditions. He proved that a
bounded operator is spectral if and only if it admits a unique decomposition
$$
 T=S+N,
$$
where $S$ is a spectral operator of scalar type, $N$ is quasinilpotent, and
$SN=NS$. See Theorem 8 in  \cite{Dunford1954}, and Theorem 5 in Chapter~XV.5 of \cite{DunfordSchwartzIII}.

This decomposition is formally similar to the Jordan--Chevalley decomposition, but the two results appeal to different theories. Their common  finite-dimensional complex case is a straightforward
specialization of each theory. The terminology
\emph{Jordan--Chevalley--Dunford decomposition}, occasionally found in the
literature, reflects the similarity of the resulting decompositions, but does not imply a common theoretical framework.
\end{remark}

\section*{Appendix: Invariant Elements in a Splitting Field}
\setcounter{theorem}{0}
\renewcommand{\thetheorem}{A\arabic{theorem}}

We prove here the elementary result on splitting fields used in the proof of Proposition~\ref{prop.rhoinK}.
\begin{proposition}\label{prop:fixed}
Let $K$ be a perfect field and let $L$ be the splitting field of a polynomial $\mu\in K[t]$. If $\alpha\in L$ satisfies $\sigma(\alpha)=\alpha$ for every $K$-automorphism $\sigma:L\to L$, then $\alpha\in K$.
\end{proposition}

We will use the following lemma.

\begin{lemma}\label{lem:fixed}
Let $K$ be a field, let $\mu\in K[t]$, and let $L$ be the splitting field of $\mu$ over $K$. Let $\Omega$ be an algebraically closed field containing $L$. If $E$ is an intermediate field $K\subseteq E\subseteq L$, then every $K$-embedding $\varphi:E\to\Omega$ extends to a $K$-embedding $\sigma:L\to\Omega$. Moreover, every such extension satisfies $\sigma(L)=L$ and is therefore a $K$-automorphism of $L$.
\end{lemma}

\begin{proof}
By the Extension Theorem for field embeddings
\cite[Chapter V, Theorem 2.8]{Lang}, the $K$-embedding
$\varphi:E\to\Omega$ extends to a $K$-embedding $\sigma:L\to\Omega$.
Since $\sigma$ fixes the coefficients of $\mu$, it maps every root of
$\mu$ to a root of $\mu$. As the roots of $\mu$ belong to $L$ and
generate $L$ over $K$, it follows that $\sigma(L)\subseteq L$.

Thus $\sigma$ may be regarded as an injective $K$-linear map
$\sigma:L\to L$. Since $\dim_K(L) < \infty$, $\sigma$ is surjective.
Therefore $\sigma(L)=L$, and $\sigma$ is a $K$-automorphism of $L$.
\end{proof}

\begin{proof}[Proof of Proposition~\ref{prop:fixed}]
Let $\theta_\alpha\in K[t]$ be the minimal polynomial of $\alpha$ over $K$. Suppose for contradiction that $\alpha\notin K$, so $\deg\theta_\alpha\geq 2$. Since $K$ is perfect, $\theta_\alpha$ is separable, hence it has another root $\beta\neq\alpha$ in an algebraic closure $\Omega$ of $L$.

Since $\alpha$ and $\beta$ are roots of the same irreducible polynomial $\theta_\alpha\in K[t]$, the evaluation maps induce isomorphisms

$$
K[t]/(\theta_\alpha)\simeq K(\alpha) \qquad\text{and}\qquad K[t]/(\theta_\alpha)\simeq K(\beta).
$$

Consequently, there exists a $K$-isomorphism $\varphi:K(\alpha)\to K(\beta)$ satisfying $\varphi(\alpha)=\beta$. Composing with the inclusion $K(\beta)\subseteq\Omega$, we regard $\varphi$ as a $K$-embedding $\varphi:K(\alpha)\to\Omega$.

By Lemma~\ref{lem:fixed}, applied with $E=K(\alpha)$, $\varphi$ extends to a $K$-automorphism $\sigma:L\to L$. In particular,

$$
\sigma(\alpha)=\beta\neq\alpha,
$$

contradicting the hypothesis. Therefore $\alpha\in K$.
\end{proof}

\end{document}